\documentclass[12pt,a4paper,reqno]{article}
 \usepackage[left=2cm, right=2cm, top=1cm,bottom=2cm]{geometry}
\usepackage[T1]{fontenc}
\usepackage{amssymb}
\usepackage{amsmath}
\usepackage{graphicx}
\usepackage{amsthm}
\usepackage{cite}
\usepackage{color}
\usepackage{esint}
\usepackage{lmodern,soul}
\usepackage{amsfonts}
\usepackage{mathrsfs}
\usepackage{hyperref}

\newcommand{\dis}[1]{$\color{red}\star$}

\newtheorem{theorem}{\textbf{Theorem}}[section]
\newtheorem{lemma}{\textbf{Lemma}}[section]
\newtheorem{proposition}{\textbf{Proposition}}[section]
\newtheorem{corollary}{\textbf{Corollary}}[section]

\numberwithin{equation}{section}

\allowdisplaybreaks[4]

\def\non{\nonumber}

\newcommand{\RR}{\mathbb{R}}
\newcommand{\A}{\mathcal{A}}
\newcommand{\E}{\mathcal{E}}
\newcommand{\TT}{\mathbb{T}}

\newcommand{\p}{\partial}
\renewcommand{\div}{\operatorname{div}}
\newcommand{\n}{\mathbf{n}}
\newcommand{\II}{\mathbb{I}_3}

\newcommand{\ud}{\mathrm{d}}
\newcommand{\tr}{\operatorname{tr}}

 \def\[{\begin{equation}}
 \def\]{\end{equation}}

\begin{document}

\title{Contact set of solutions to  gradient flow of Landau-de Gennes energy with a singular entropy potential}

\author{
Yuning Liu\footnote{NYU Shanghai, 567 Yangsi W road, Pudong, Shanghai 200126, China, and NYU-ECNU Institute of
Mathematical Sciences at NYU Shanghai, 3663 Zhongshan Road North, Shanghai, 200062, China
  \texttt{yl67@nyu.edu}} \,
  \and
  Xiang Xu\footnote{
  Department of Mathematics and Statistics,
  Old Dominion University, Norfolk, VA 23529, USA.
  \texttt{x2xu@odu.edu}}
 }

\date{}
\maketitle
 
\begin{abstract} 
This paper investigates the convexity structure and contact set of the $Q$-tensor flow with  anisotropic Landau--de Gennes elastic energy and    a singular entropy potential $f(Q)$. The flow naturally exhibits a coercivity property, yielding the $H^1$ regularity of $f(Q)$ for almost every time. Combining this $H^1$ regularity with spherical averaging and capacity theory, the contact set
$$\mathcal{C}(Q)\triangleq \{x\mid f(Q(x))=\infty\}$$
is characterized. In particular, it is shown to be empty in two spatial dimensions and to have Hausdorff dimension at most one in three spatial dimensions.

\end{abstract}


\section{Introduction}
The Landau-de Gennes theory is a continuum theory for  nematic liquid crystals.
When formulating static or dynamic continuum theories, a crucial step is to select an order parameter that captures the microscopic
scale structure in rod-like molecule systems.  In our framework, the order parameter is a matrix-valued function that takes values in the following so-called $Q$-tensor space:
\begin{equation}\label{Q-tensor-space}
\mathcal{Q}\triangleq \left\{M\in\RR^{3\times 3}\big| \tr M=0; \,M=M^T\right\}.
\end{equation}

To establish the functional setting early on, we introduce the physical $Q$-tensor space and the fundamental norms here. Parallel to the configuration space \eqref{Q-tensor-space}, we introduce the \textit{physical} $Q$-tensor space by:
\begin{equation}\label{q-tensor}
  \mathcal{Q}_{phy} \triangleq \Big\{M\in\RR^{3\times 3}\big| \tr M=0; \,M=M^T;\,  -\frac13< \lambda_1(M)\leq \lambda_2(M)\leq \lambda_3(M)< \frac23\Big\}.
\end{equation}
Any element in $\mathcal{Q}_{phy}$ is called a physical $Q$-tensor. For any $Q, R\in\mathcal{Q}$, $Q:R$ stands for $\tr(QR)$, and $\|Q\|_{F}=\sqrt{\tr(Q^tQ)}$ represents the Frobenius norm of $Q$.   

In the computational domain $\TT^n$,  which represents the unit square in $\mathbb{R}^n$ with $n=2$ or $3$, the order parameter $Q: \TT^n\to \mathcal{Q}$ is used to describe the static or dynamic configuration of the liquid crystals material. The associated free energy functional $\E$ consists of the elastic
and the bulk parts, which reads:
\begin{equation}\label{free-energy}
\E(Q) \triangleq  \mathcal{G}(\nabla Q)+\mathcal{B}(Q)-\alpha\|Q\|^2_{L^2(\TT^n)}.
\end{equation}
Here $\mathcal{G}$ stands for the anisotropic elastic
energy that contains three quadratic terms of the gradient of $Q$:
\begin{equation}\label{elastic-energy}
\mathcal{G}(\nabla Q)\triangleq  \frac 12 \int_{\TT^n}\big(L_1|\nabla Q|^2+L_2\p_j Q_{ik}\p_k Q_{ij}+L_3\p_j Q_{ij}\p_k Q_{ik}\big)\,\ud x
\end{equation}
where $L_1, L_2, L_3$ are material-dependent constants. Here and in the sequel, $\partial_kQ_{ij}$ denotes the $k$-th spatial partial derivative of the $ij$-th component
of $Q$, and we adopt the Einstein
summation convention by summing over repeated Latin letters.
Following \cite{MR3927580}, we shall assume
\begin{equation}\label{coefficient-assumption}
L_1>0,\quad L_1>|L_2+L_3|,
\end{equation}
which ensures that \eqref{elastic-energy} fulfills the strong Legendre condition.

In \eqref{free-energy}, $\mathcal{B}(Q)$ denotes the bulk energy:
\begin{equation}\label{bulk-energy}
  \mathcal{B}(Q)\triangleq \int_{\TT^n} f(Q)\,\ud{x},
\end{equation}
where the integrand $f(Q)$ is the singular potential introduced in \cite{maiersaupe}:
\begin{equation}\label{def-singluar-potential}
f(Q)\triangleq \begin{cases}
\displaystyle\inf_{\rho\in\mathcal{A}_{Q}}\int_{\mathbb{S}^2}\rho(\n)\ln{\rho(\n)}\,\ud \n,
\quad&\mbox{if } -\frac13<\lambda_i(Q)<\frac23,\; 1\leq i\leq 3,\\
+\infty, \quad&\mbox{otherwise}.
\end{cases}
\end{equation}
In \eqref{def-singluar-potential}, $\lambda_i(Q)$ denotes the $i$-th eigenvalue of the matrix $Q$, ordered non-decreasingly, and $\mathcal{A}_{Q}$ is the admissible class defined by:
\begin{equation}\label{admissible-set}
  \mathcal{A}_{Q}=\left\{\rho(\n):\mathbb{S}^2\rightarrow \mathbb{R}_{\geq 0}\mid \,
  \|\rho\|_{L^1(\mathbb{S}^2)}=1;\ \rho(\n)=\rho(-\n);\
   \int_{\mathbb{S}^2}\big(\n\n^T-\frac13\II\big)\rho(\n)\,\ud{\n}=Q\right\}.
\end{equation}
It is noted that this singular potential \eqref{def-singluar-potential} imposes physical constraints on the eigenvalues of $Q$.
Further, $\alpha>0$ in \eqref{free-energy} is a temperature-dependent constant which characterizes the relative intensity of the molecular Brownian motion and the molecular interaction. 

The singular potential possesses several important analytic properties, which have been   studied in previous works including, among others, \cite{maiersaupe,MR3383939,ballpassage,MR3729353,MR4346781}. We refer the reader to the survey \cite{MR3729353,MR4410922} and the references therein.

This paper now turns to a rigorous study of the gradient flow generated by $\E(Q)$:
 \begin{equation}\label{grad flow E}
\partial_tQ(t,\cdot)=-\delta\E(Q(t,\cdot)), \quad\forall t>0
\end{equation}
subject to periodic boundary conditions:
\begin{equation}\label{BC-periodic}
Q(t, x+e_i)=Q(t, x), \quad\mbox{for } (t, x)\in\mathbb{R}^+\times \TT^n.
\end{equation} 
Here, $\delta\E(Q(t,\cdot))$ is the variation of $\E$, which can be written in component form as:
 \begin{align}\label{strong solu}
-(\delta\E(Q))_{ij}&= L_1\Delta{Q}_{ij}+ (L_2+L_3) \p_j\p_kQ_{ik}-\dfrac{  L_2+L_3}{3}\p_k\p_
{\ell}Q_{\ell k}\delta_{ij}\non\\
&\qquad-\frac{\partial f}{\partial{Q}_{ij}}+\frac13\tr\Big(\frac{\partial  f}{\partial{Q}}\Big)\delta_{ij}+2\alpha Q_{ij}.
\end{align}

 We begin by stating the main existence and regularity theorem for the gradient flow. 
 
\begin{theorem} \label{main-theorem-1}
Let $n=2$ or $3$. For any initial data $Q_0 \in H^1(\TT^n;\mathcal{Q}_{phy})$ with  
\[\mathcal{E}(Q_0)+\|\delta \mathcal{E}(Q_0)\|_{L^2} <\infty,\]  
there exists a unique global solution $Q(t,x):\RR^+\times\TT^n\to \mathcal{Q}_{phy}$ of \eqref{grad flow E} such that \eqref{strong solu} holds almost everywhere in $(0,T)\times \TT^n$. Moreover, for each fixed $T>0$, there holds
\begin{subequations}
\begin{align}
&Q\in L^\infty(0,T; H^2(\TT^n)) \quad \&\quad  \partial_t Q\in L^\infty(0,T; L^2(\TT^n)) \cap L^2(0,T; H^1(\TT^n)),\label{infty2-solu}\\
&\int_0^T\int_{\TT^n} \left[D^2_Q f(Q) : (\nabla Q)^{\otimes 2} +D^2_Q f(Q) : (\p_t Q)^{\otimes 2}\right] \,\ud x \ud t <\infty. \label{hessian-space}
\end{align}
\end{subequations}
 Moreover, we have 
  \begin{equation} 
 f(Q)\in L^2(0,T; H^1(\TT^n)).\label{wellposed1}
 \end{equation}
\end{theorem}

 In \cite{MR4272911} a strong solution to \eqref{grad flow E}  has been  constructed, but higher-order energy estimates analogous to \eqref{hessian-space}   were not obtained. These estimates, together with the a coercivity inequality (cf. \eqref{self-concordant} below), lead to \eqref{wellposed1}, which is crucial for proving partial regularity and analyzing the contact  set of solutions.
 
 The following result, which is concerned with 
 the constraints on the size of the contact set,  does not depend on the specific PDE dynamics, but solely on the  $H^1$ regularity of $f(Q)$ and  the spatial continuity of $Q$. Therefore, we formulate it  as a  standalone theorem, which may be  of independent interest.

\begin{theorem} \label{thm-contact-set}
 Let $\Omega \subset \mathbb{R}^n$ be a bounded domain with dimension $n \in \{2,3\}$. Suppose 
\[Q \in H^2(\Omega; \overline{\mathcal{Q}_{phy}})\quad \text{ and }\quad f(Q) \in H^1(\Omega).\label{contact2}\]
 Then  the contact set  
\[\mathcal{C}(Q) \triangleq \{x\in \Omega \mid f(Q(x))=\infty\}\]
satisfies the following properties:
\begin{itemize}
    \item When $n=2$,  the contact set is   empty, $\mathcal{C}(Q) = \emptyset$.
    \item When $n=3$, then  the contact set has Hausdorff dimension at most 1. More precisely, $\mathcal{H}^{1+\epsilon}(\mathcal{C}(Q)) = 0$ for any $\epsilon > 0$.
\end{itemize}
\end{theorem}

As a consequence of these two theorems, we have the following improved result on the contact set in \cite[Theorem 1.3]{MR4272911}. 
 \begin{corollary}
 Let $Q(t, x)$ be the unique   solution   established in Theorem \ref{main-theorem-1}. Then 
 for a.e. $t \in\left(0, T\right)$, the contact set
has the following properties:
\begin{itemize}
\item $\mathcal{C}(Q(t,\cdot))=\emptyset$ for $n=2$.
\item $\mathcal{C}(Q(t,\cdot))$ has Hausdorff dimension at most 1 for $n=3$.
\end{itemize}
 \end{corollary}
 The investigation of the contact set for variational problems involving singular energy functionals of Landau–de Gennes type originated in the work of Evans--Kneuss--Hung \cite{MR3451881} (see also \cite{MR4073203}). There, the authors considered a general elastic energy density $\mathcal{G}(\nabla Q, Q)$ that is uniformly quasiconvex with respect to $\nabla Q$, establishing partial regularity for energy minimizers. In the special case where $\mathcal{G}$ is strictly convex and depends solely on $\nabla Q$—which aligns closely with the setup considered in the present work (cf. \eqref{elastic-energy})—they obtained $H^2$-regularity for energy minimizers and established the Hausdorff measure bound $\mathcal{H}^{n-2+\epsilon}(\mathcal{C}(Q))=0$, under the additional hypothesis that
 \begin{equation}\label{evans}f(Q) \geq \frac{\gamma}{\operatorname{dist}(Q, \partial \mathbb{K})^2}, \quad (Q \in \mathbb{K}).\end{equation}
 Here  $\mathbb{K}$ denotes the physical domain such as   $\mathcal{Q}_{\text{phy}}$ defined in \eqref{q-tensor}. However, the physical singular potential \eqref{def-singluar-potential} exhibits a logarithmic divergence near $\partial \mathbb{K}$, which is significantly weaker than the polynomial blow-up required by \eqref{evans}. To extend their framework to more physically realistic potentials, the authors in \cite{MR3451881} proposed assuming the following self-concordant-like coercivity inequality for $f$:
 $$C\vert{}y\vert{}^2 + y^T D^2 f(Q) y \geq \gamma \vert{}D f(Q) \cdot y\vert{}^2, \quad (Q \in \mathbb{K}, \; y \in \mathbb{R}^k).$$
  In  \cite{MR3383939},  the authors  have   successfully established  that the singular logarithmic potential \eqref{def-singluar-potential} indeed satisfies this coercivity inequality through a careful  analysis of the hessian of $f$. 
  In Lemma \ref{Self-concordance-thm} below, we employ a dual formulation to provide a new proof of this inequality, featuring a more streamlined computation of the integrals.
  Leveraging this key coercivity inequality, we are able to generalize the contact set estimate $\mathcal{H}^{n-2+\epsilon}(\mathcal{C}(Q))=0$ to the dynamic evolutionary system, where static variational techniques for energy minimizers are no longer applicable.
 
 
 \section{Blow-up rate and convexity of $f$}
 
 We    recall   a few  existing  results for the singular potential $f$ defined in \eqref{def-singluar-potential}.

\begin{proposition}\label{maiersaupe1}
\cite{maiersaupe}
Given $Q\in \mathcal{Q}_{phy}$, there exist unique $\mu_1,\mu_2,\mu_3$ such that the optimal density function $\rho_Q\in\A_{Q}$ that solves the minimizing problem defining $f(Q)$ in \eqref{def-singluar-potential} is given by 
\[\begin{aligned}\label{Boltzmann-distribution}
&\rho_Q(x,y,z)=\dfrac{\exp(\mu_1x^2+\mu_2y^2+\mu_3z^2)}{Z(\mu_1,\mu_2,\mu_3)}, \quad(x,y,z)\in\mathbb{S}^2 ,\\&\quad\text{ with }\mu_1+\mu_2+\mu_3=0.
\end{aligned}\]
Here in \eqref{Boltzmann-distribution}, $Z(\mu_1,\mu_2,\mu_3)$ is the partition function given by
\begin{equation}\label{partition6}
Z(\mu_1,\mu_2,\mu_3)=\int_{\mathbb{S}^2 }\exp(\mu_1x^2+\mu_2y^2+\mu_3z^2)\,\ud{S},
\end{equation}
which satisfies the moment condition:
\begin{equation}\label{second-moment-original}
\frac{1}{Z}\frac{\partial Z}{\partial\mu_i}=\int_{\mathbb{S}^2} n_i^2 \rho_Q(\mathbf{n}) \, \ud\mathbf{n} =\lambda_i+\frac13, \qquad 1\leq i\leq 3.
\end{equation}
\end{proposition}

We also need   a monotonicity property between the eigenvalues of $Q$ and the associated Lagrange multipliers.

\begin{proposition}\cite{MR4346781}\label{lemma-orders-mu}
For any physical $Q$-tensor \eqref{Q-diagonal}, its optimal probability density $\rho_Q$ defined in \eqref{Boltzmann-distribution} satisfies
$$
  \mu_1\leq\mu_2\leq\mu_3.
$$
Moreover, $\mu_i=\mu_j$ provided $\lambda_i=\lambda_j$ for $1\leq i\neq j\leq 3$.
\end{proposition}


%

 \begin{proposition}\cite{MR4346781}\label{theorem-upper-bound}
On the domain $ \mathcal{Q}_{phy}$ \eqref{q-tensor}, the functional $f$ defined in \eqref{def-singluar-potential} is bounded above by
\begin{equation}\label{upper-bound-new}
f(Q)\leq -\frac12\ln\Big(\lambda_1(Q)+\frac13\Big)-\frac12\ln\Big(\lambda_2(Q)+\frac13\Big)-\ln{8\sqrt{3}}.
\end{equation}
Moreover, there exists a small computable constant $\delta_0>0$, such that, whenever  $\lambda_2(Q)+1/3<\delta_0$, it holds
\begin{equation}\non
f(Q)\geq -\frac12\ln\Big(\lambda_1(Q)+\frac13\Big)-\frac12\ln\Big(\lambda_2(Q)+\frac13\Big) - 1 - 3\ln\pi. 
\end{equation}
\end{proposition}
 
Both bounds in Proposition \ref{theorem-upper-bound} are rigorously proved in \cite{MR4346781}. In particular, the proof of the  lower bound in \cite{MR4346781} is quite    delicate and complex. Due to its importance for our main Theorem \ref{thm-contact-set}, we shall provide  a shorter proof via  a dual formulation proposed in \cite{MR3423248}. 
Note that our argument does not rely on \eqref{upper-bound-new}.

  Since $f$ is rotation invariant (cf. \cite{maiersaupe}),   we   assume in this section  that any considered physical $Q$-tensor is diagonal:
\begin{equation}\label{Q-diagonal}
Q={\rm diag} (
      \lambda_1 ,\lambda_2,\lambda_3), \qquad -\frac13<\lambda_1\leq\lambda_2\leq\lambda_3<\frac23,\;\lambda_1+\lambda_2+\lambda_3=0.
\end{equation}

 \subsection{Dual formulation and  lower bound in Proposition \ref{theorem-upper-bound}}
  
We shall establish the lower bound in Proposition \ref{theorem-upper-bound} using a dual argument. For this purpose it is convenient to reparameterize the problem in terms of the distance to the physical boundary. Let $Q\in\mathcal{Q}_{phy}$ have eigenvalues $\lambda_1, \lambda_2, \lambda_3$. We define the shifted eigenvalues representing this distance:
\begin{equation}\label{def-delta}
\delta_1 = \lambda_1(Q) + \frac13, \qquad \delta_2 = \lambda_2(Q) + \frac13.
\end{equation}
Since $\tr Q = 0$, the third shifted eigenvalue is implicitly given by $\lambda_3 + \frac{1}{3} = 1 - \delta_1 - \delta_2$.  

Recall from Proposition \ref{maiersaupe1} that the optimal density function solving \eqref{def-singluar-potential} is
\begin{equation}\label{rhoQ}
\rho_Q(\mathbf{n}) = \frac{1}{Z(\mu_1, \mu_2, \mu_3)} \exp(\mu_1 x^2 + \mu_2 y^2 + \mu_3 z^2),
\end{equation}
subject to the moment constraints \eqref{second-moment-original}, where the full partition function is:
\begin{equation}\label{partition5}
Z(\mu_1, \mu_2, \mu_3) = \int_{\mathbb{S}^2} \exp(\mu_1 x^2 + \mu_2 y^2 + \mu_3 z^2) \,\ud\mathbf{n}.
\end{equation}

Because the parameters are subject to the gauge constraint $\mu_1 + \mu_2 + \mu_3 = 0$ and the geometric constraint $x^2+y^2+z^2=1$, we introduce a set of strictly independent dual variables:
\begin{equation}\label{def-nu}
\nu_1 = -(2\mu_1+\mu_2),\qquad \nu_2 = -(\mu_1+2\mu_2).
\end{equation}
Using $\mu_3 = -\mu_1 - \mu_2$ and $z^2 = 1 - x^2 - y^2$, we find that the exponent can be rewritten as:
\begin{equation}\label{partition1}
\begin{aligned}
&\mu_1 x^2 + \mu_2 y^2 + \mu_3 z^2 \\
&= \mu_1 x^2 + \mu_2 y^2 - (\mu_1 + \mu_2)(1 - x^2 - y^2) \\
&= -\nu_1 x^2 - \nu_2 y^2 - (\mu_1 + \mu_2).
\end{aligned}
\end{equation}
This motivates the definition of a reduced partition function $\tilde{Z}(\nu_1, \nu_2)$, which depends only on the independent variables:
\begin{equation}\label{reduced-Z}
\tilde{Z}(\nu_1, \nu_2) \triangleq  \int_{\mathbb{S}^2} \exp(-\nu_1 x^2 - \nu_2 y^2) \,\ud\mathbf{n}.
\end{equation}
Consequently, the full partition function decomposes cleanly as:
\begin{equation}\label{partition4}
Z(\mu_1, \mu_2, \mu_3) = \exp(-\mu_1 - \mu_2) \tilde{Z}(\nu_1, \nu_2).
\end{equation}

Evaluating the entropy potential $f(Q)$ at the optimal distribution $\rho_Q$, we expand the logarithmic term using \eqref{second-moment-original}:
\begin{align*}
f(Q) &= \int_{\mathbb{S}^2} \rho_Q \ln \rho_Q \,\ud\mathbf{n} = \int_{\mathbb{S}^2} \rho_Q \Big[ (\mu_1 x^2 + \mu_2 y^2 + \mu_3 z^2) - \ln Z \Big] \,\ud\mathbf{n}\\
&= \sum_{i=1}^3 \Big( \mu_i  \int_{\mathbb{S}^2} n_i^2 \rho_Q \,\ud\mathbf{n}  \Big) - (\ln Z)  \int_{\mathbb{S}^2} \rho_Q \,\ud\mathbf{n} \\
&= \sum_{i=1}^3   \mu_i  \lambda_i  - \ln Z.
\end{align*}
By substituting our shifted eigenvalues \eqref{def-delta} and the decomposed partition function \eqref{partition4}, we find:
\begin{align}
f(Q) &\overset{\eqref{def-delta}}= \mu_1 \delta_1 + \mu_2 \delta_2 + \mu_3(1 - \delta_1 - \delta_2) - \ln Z \nonumber \\
&= \delta_1(2\mu_1 + \mu_2) + \delta_2(\mu_1 +2 \mu_2) - (\mu_1 + \mu_2) - \Big[ -(\mu_1 + \mu_2) + \ln \tilde{Z}(\nu_1, \nu_2) \Big] \nonumber \\
&= -\delta_1\nu_1 - \delta_2\nu_2 - \ln \tilde{Z}(\nu_1, \nu_2). \label{entropy-exact}
\end{align}
Note that $\nu_1, \nu_2$ must satisfy the integral constraints implied by \eqref{def-delta}:
\begin{equation}\label{integral-constraints}
\begin{aligned}
\delta_1 &= \int_{\mathbb{S}^2 } x^2 \rho_Q \,\ud\mathbf{n}
= \frac{\int_{\mathbb{S}^2 } x^2 \exp(-\nu_1 x^2 - \nu_2 y^2) \,\ud\mathbf{n}}{\int_{\mathbb{S}^2 } \exp(-\nu_1 x^2 - \nu_2 y^2) \,\ud\mathbf{n}}, \\
\delta_2 &= \int_{\mathbb{S}^2 } y^2 \rho_Q \,\ud\mathbf{n}
= \frac{\int_{\mathbb{S}^2 } y^2 \exp(-\nu_1 x^2 - \nu_2 y^2) \,\ud\mathbf{n}}{\int_{\mathbb{S}^2 } \exp(-\nu_1 x^2 - \nu_2 y^2) \,\ud\mathbf{n}}.
\end{aligned}
\end{equation}

\begin{lemma} [Dual formulation]
The parameters $(\nu_1, \nu_2)$ satisfying \eqref{integral-constraints} are strictly positive (provided the state is not completely isotropic), and they are the unique global maximizers of the function $J : \mathbb{R}^2\to \mathbb{R}$ defined by:
\[\label{dual-formula}
J(\nu_1, \nu_2) \triangleq  -\delta_1\nu_1 - \delta_2\nu_2 - \ln \tilde{Z}(\nu_1, \nu_2),
\]
where $\tilde{Z}$ is given in \eqref{reduced-Z}. Thus, $f(Q)$ can be formulated    as an unconstrained maximization problem:
\begin{equation}\label{maiersaupe3}
f(Q) = \sup_{\nu_1, \nu_2 \in\mathbb{R}} J(\nu_1, \nu_2). 
\end{equation}
\end{lemma}

\begin{proof}
First, we verify the non-negativity of $\nu_i$. From Lemma \ref{lemma-orders-mu}, we know that $\mu_1 \leq \mu_2 \leq \mu_3$. Combined with the zero-trace condition $\mu_1 + \mu_2 + \mu_3 = 0$, this implies $\mu_1 \leq 0$ and $\mu_3 \geq 0$. Now, observing the definitions of $\nu_i$ in \eqref{def-nu}, we have $\nu_1 = \mu_3 - \mu_1 \geq 0$ and $\nu_2 = \mu_3 - \mu_2 \geq 0$. As long as the state is strictly inside the physical domain and not completely isotropic ($Q \neq 0$), we have $\mu_1 < \mu_3$, which immediately implies $\nu_1 > 0$. While $\nu_2$ can be zero for uniaxial states where $\lambda_2 = \lambda_3$, the function $J$ remains well-defined and continuous for $\nu_2 \geq 0$.

 Next, taking the first partial derivative of $J$ with respect to $\nu_1$ gives:
\begin{align*}
\frac{\partial J}{\partial \nu_1} = -\delta_1 - \frac{\partial}{\partial \nu_1} \ln \tilde{Z}(\nu_1, \nu_2)\overset{\eqref{reduced-Z}} = -\delta_1 + \frac{\int_{\mathbb{S}^2} x^2 \exp(-\nu_1 x^2 - \nu_2 y^2) \,\ud\mathbf{n}}{\int_{\mathbb{S}^2} \exp(-\nu_1 x^2 - \nu_2 y^2) \,\ud\mathbf{n}}.
\end{align*}
Setting $\frac{\partial J}{\partial \nu_1} = 0$ exactly recovers the first integral constraint for $\delta_1$ in \eqref{integral-constraints}. By exact symmetry, setting $\frac{\partial J}{\partial \nu_2} = 0$ yields the second integral constraint for $\delta_2$. The exact Lagrange multipliers $(\nu_1^*, \nu_2^*)$ are thus precisely the critical points where $\nabla J = 0$.
  
To prove that such a  critical point is a global supremum, we look at the Hessian matrix of $J$. Let 
\begin{equation}\label{density2}
\rho(\nu_1, \nu_2) = \frac{\exp(-\nu_1 x^2 - \nu_2 y^2)}{\int_{\mathbb{S}^2} \exp(-\nu_1 x^2 - \nu_2 y^2) \,\ud\mathbf{n}}
\end{equation}
be the probability density for a given set of parameters. We already showed $\frac{\partial J}{\partial \nu_1} = -\delta_1 + \langle x^2 \rangle_\rho$. Taking the second derivative $\frac{\partial^2 J}{\partial \nu_1^2}$ means taking the derivative of the expectation $\langle x^2 \rangle_\rho$ with respect to $\nu_1$. Through standard properties of exponential families, the derivative of an expected value with respect to its natural parameter is its variance:
\begin{equation}\non
\frac{\partial^2 J}{\partial \nu_1^2} = -\Big( \langle x^4 \rangle_\rho - \langle x^2 \rangle_\rho^2 \Big) = -\text{Var}_\rho(x^2).
\end{equation}
In order for $\text{Var}_\rho(x^2)=0$ we must have $x^2=const$ almost everywhere  with respect to $\rho(\nu_1, \nu_2)$, which is impossible given the exponential form of $\rho$.
Similarly, $\frac{\partial^2 J}{\partial \nu_2^2}= -\text{Var}_\rho(y^2)$, and  the mixed partial derivative is the negative covariance:
\begin{equation}\non
\frac{\partial^2 J}{\partial \nu_1 \partial \nu_2} = -\Big( \langle x^2 y^2 \rangle_\rho - \langle x^2 \rangle_\rho \langle y^2 \rangle_\rho \Big) = -\text{Cov}_\rho(x^2, y^2).
\end{equation}
This means the Hessian matrix   $H(J)$ is exactly the negative covariance matrix of the random vector $(x^2, y^2)$.

 To ensure $H(J)$ is strictly negative definite, we must verify that the covariance matrix is strictly positive definite. A covariance matrix is singular   if and only if there exists an exact linear relationship between the random variables; that is, if there exist constants $c_1, c_2, c_3$ (not all zero) such that $c_1 x^2 + c_2 y^2 = c_3$ almost everywhere with respect to $\rho$, cf. \cite[Section 3.6]{MR4229142}. Since $x^2 + y^2 + z^2 = 1$ on $\mathbb{S}^2$, this condition implies $c_1 x^2 + c_2 y^2 = c_3 (x^2 + y^2 + z^2)$ on the sphere. Because the continuous polynomials $x^2, y^2,$ and $z^2$ are linearly independent over $\mathbb{S}^2$, this identity forces $c_1 = c_3$, $c_2 = c_3$, and $0 = c_3$, which means $c_1 = c_2 = c_3 = 0$. This contradicts the assumption that the constants are not all zero. Therefore, $x^2$ and $y^2$ are   linearly independent as functions on the sphere. Because the density $\rho$ is strictly positive everywhere, the variance of any non-trivial linear combination is strictly positive. Consequently, the covariance matrix is strictly positive definite, making the Hessian matrix $H(J)$ strictly negative definite. Since $J$ is strictly concave over $\mathbb{R}^2$, the unique critical point is the global supremum. 
\end{proof}

 As an immediate application of the dual formulation \eqref{maiersaupe3}, we can easily recover the well-known Bingham closure, which characterizes $f(Q)$ as a Legendre transform.

\begin{corollary} \cite{MR3423248}
 $f(Q)$ is the Legendre transform of the function $ \ln Z(M)$ where 
\begin{equation}\label{partition2}
   Z(M)\triangleq  \int_{\mathbb{S}^2} e^{M : \mathbf{n} \otimes \mathbf{n}} \,\ud\mathbf{n} : \mathcal{Q}\to \mathbb{R}.
\end{equation}
More precisely, 
\begin{equation}\label{sup-problem}
f(Q) = \sup_{M \in\mathcal{Q}}\left(M : Q-\ln Z(M)\right).
\end{equation}
\end{corollary}

\begin{proof}
 Because the partition function $Z(M )$ is isotropic, the supremum in \eqref{sup-problem} is achieved when $M $ and $Q$ share the same eigenframe. Therefore, to compute the right-hand side, it suffices to consider diagonal matrices:
$$M ={\rm diag}\{\mu_1,\mu_2,\mu_3\}\in\mathcal{Q}\quad \text{and}\quad Q={\rm diag}\{\lambda_1,\lambda_2,\lambda_3\} \in\mathcal{Q}_{phy}.$$ 

 First, we evaluate the inner product $M  : Q$ using the shifted eigenvalues defined in \eqref{def-delta}:
\begin{align*}
M  : Q =& \mu_1 \Big(\delta_1 - \frac{1}{3}\Big) + \mu_2 \Big(\delta_2 - \frac{1}{3}\Big) + \mu_3 \Big(1 - \delta_1 - \delta_2 - \frac{1}{3}\Big)\\
  =& \mu_1 \delta_1 + \mu_2 \delta_2 + \mu_3 (1 - \delta_1 - \delta_2) - \frac{1}{3}\underbrace{(\mu_1 + \mu_2 + \mu_3)}_{=0}\\
  =& \mu_1 \delta_1 + \mu_2 \delta_2 - (\mu_1 + \mu_2) + \delta_1(\mu_1 + \mu_2) + \delta_2(\mu_1 + \mu_2) \\
   =& \delta_1(2\mu_1 + \mu_2) + \delta_2(\mu_1 + 2\mu_2) - (\mu_1 + \mu_2)\\
   \overset{\eqref{def-nu}}=& -\delta_1 \nu_1 - \delta_2 \nu_2 - (\mu_1 + \mu_2).
\end{align*}

We can use \eqref{partition5}
 and  \eqref{partition4} to write:
\begin{equation}\non
  \ln Z(\mu_1, \mu_2, \mu_3) = -(\mu_1 + \mu_2) + \ln \tilde{Z}(\nu_1, \nu_2).
\end{equation}

 Subtracting the two expressions cancels the $(\mu_1 + \mu_2)$ terms, yielding:
\begin{align*}
M  : Q -   \ln Z  = -\delta_1 \nu_1 - \delta_2 \nu_2 - \ln \tilde{Z}(\nu_1, \nu_2).
\end{align*}
 The right-hand side is exactly the   function $J(\nu_1, \nu_2)$ defined in \eqref{dual-formula}. Taking the supremum over all admissible $M $ corresponds to taking the supremum over $\nu_1, \nu_2 \in \mathbb{R}$, so the result follows directly from \eqref{maiersaupe3}.
\end{proof}

 Finally, we utilize the unconstrained dual formulation \eqref{maiersaupe3} to provide a direct   proof of the lower bound stated in Proposition \ref{theorem-upper-bound}:

\begin{corollary} 
There exists a small computable constant $\delta_0>0$ such that, whenever $Q\in \mathcal{Q}_{phy}$  with $\lambda_2(Q)+1/3<\delta_0$, it holds that:
\begin{equation}\label{lower-bound-new}
f(Q)\geq -\frac12\ln\Big(\lambda_1(Q)+\frac13\Big)-\frac12\ln\Big(\lambda_2(Q)+\frac13\Big) - 1 - 3\ln\pi. 
\end{equation}
\end{corollary}

\begin{proof}
Recall \eqref{maiersaupe3} and \eqref{dual-formula}.
To establish a lower bound for  $f(Q) = \sup J(\nu_1, \nu_2)$, it suffices to derive  an upper bound for the reduced partition function $\tilde{Z}(\nu_1, \nu_2)$. Using standard spherical coordinates 
\begin{equation}\label{spherical-coord}
x=\sin\theta\cos\phi, \quad y=\sin\theta\sin\phi, \quad z=\cos\theta,
\end{equation}
and denoting $A(\phi) = \nu_1\cos^2\phi + \nu_2\sin^2\phi$, we have:
\begin{equation}\non
\tilde{Z}(\nu_1, \nu_2) = 2 \int_0^{2\pi} \int_0^{\frac{\pi}{2}} \sin\theta \exp\big(-A(\phi)\sin^2\theta\big) \,\ud\theta \,\ud\phi.
\end{equation}
Applying the inequality $\theta\geq \sin\theta \geq \frac{2}{\pi}\theta$ for $\theta \in [0, \frac{\pi}{2}]$, we can bound the inner integral:
\begin{align*}
\tilde{Z}(\nu_1, \nu_2) &\leq 2 \int_0^{2\pi} \int_0^{\frac{\pi}{2}} \theta \exp\Big(-A(\phi)\frac{4}{\pi^2}\theta^2\Big) \,\ud\theta \,\ud\phi \\
&\leq 2 \int_0^{2\pi} \left[ \frac{\pi^2}{8A(\phi)} \Big(1 - \exp\big(-A(\phi)\big) \Big) \right] \ud\phi \\
&\leq \frac{\pi^2}{4} \int_0^{2\pi} \frac{1}{\nu_1\cos^2\phi + \nu_2\sin^2\phi} \,\ud\phi.
\end{align*}
This standard elliptic integral evaluates precisely to $\frac{2\pi}{\sqrt{\nu_1\nu_2}}$. Thus, we obtain the rigorous upper bound:
\begin{equation}\label{Z-bound}
\tilde{Z}(\nu_1, \nu_2) \leq \frac{\pi^3}{2\sqrt{\nu_1\nu_2}}.
\end{equation}
Substituting the bound \eqref{Z-bound} into \eqref{entropy-exact} gives  a global lower bound for   $J$:
\begin{equation}\non
f(Q) \geq -\delta_1\nu_1 - \delta_2\nu_2 + \frac{1}{2}\ln \nu_1 + \frac{1}{2}\ln\nu_2 - \ln\Big(\frac{\pi^3}{2}\Big), \qquad \forall \nu_1, \nu_2 > 0.
\end{equation}
To maximize this lower bound, we optimize the test multipliers by taking the partial derivatives with respect to $\nu_1$ and $\nu_2$ and setting them to zero. This yields the optimal asymptotic selections
$\nu_1 = \frac{1}{2\delta_1},  \nu_2 = \frac{1}{2\delta_2}.$
Plugging these choices back into the inequality, we obtain:
\begin{align*}
f(Q)  \geq    -1 + \frac{1}{2}\ln\Big(\frac{1}{4\delta_1\delta_2}\Big) - 3\ln\pi + \ln 2.
\end{align*} 
In view of the definitions in \eqref{def-delta},  this implies the desired lower bound.
\end{proof}


\subsection{Self-concordance of $f$}

The following important result has been established in \cite{MR3383939} by a technical integral calculation. Using dual formula \eqref{sup-problem} we can give a simpler proof.
 \begin{lemma}\label{Self-concordance-thm}
The  singular potential \eqref{def-singluar-potential} satisfies
 \begin{equation}\label{self-concordant}
A : D^2_Q f(Q) A \geq   C  (D_Q f(Q) : A)^2 \qquad \forall Q\in\mathcal{Q}_{phy}, \forall A\in\mathcal{Q}.
 \end{equation}
 for a  constant $C>0$ that is independent of $Q$ and $A$.
\end{lemma}

\begin{proof}[Proof of Lemma \ref{Self-concordance-thm}]

 \textbf{Reduction by Legendre transformation:}

 Recall from \eqref{partition6} that the partition function can be generalized for any matrix $M\in\mathcal{Q}$:
 \begin{equation}\label{partition7}
 Z(M) = \int_{\mathbb{S}^2} e^{M : \mathbf{n} \otimes \mathbf{n}} \,\ud\mathbf{n}.
 \end{equation}
 Note that we do not assume $M \in \mathcal{Q}$ is diagonal. Assume that the supremum in \eqref{sup-problem} is achieved by the optimal dual tensor $M_Q \in \mathcal{Q}$. By the standard properties of the Legendre transform \eqref{partition2} (cf. \cite{MR690288,Villani2003}), we have the following dual relations for each $Q\in\mathcal{Q}_{phy}$:
 \begin{subequations}
 \begin{align}
 M_Q &=  D_Q f(Q),\\
 Q &=  D_M \ln Z(M_Q),\\
 D^2_Q f(Q) &=  \Big[ D^2_M \ln Z(M_Q) \Big]^{-1}.
 \end{align}
  \end{subequations}

 Let us denote the Hessian of the log-partition function evaluated at the optimum as $$\mathcal{H} \triangleq D^2_M \ln Z(M_Q).$$   Then the inequality \eqref{self-concordant} that  we want to prove  can be rewritten as:
\begin{equation}\label{partition8}
A : \mathcal{H}^{-1} A \geq C (M_Q : A)^2\qquad \forall A\in\mathcal{Q}.
\end{equation}
 Recall from \eqref{partition7} that the first derivative is an expectation:
\begin{equation}\label{partition9}
D_M \ln Z(M) = \frac{1}{Z(M)} \int_{\mathbb{S}^2}  \mathbf{n} \otimes \mathbf{n}\ e^{M : \mathbf{n} \otimes \mathbf{n}} \,\ud\mathbf{n} =  \langle \mathbf{n} \otimes \mathbf{n} \rangle_{\rho_M}. 
\end{equation}
where 
\begin{equation}\non
\rho_M(\n)\triangleq \frac{1}{Z(M)}  e^{M : \mathbf{n} \otimes \mathbf{n}}. \end{equation}
Taking the derivative of \eqref{partition9} yields the Hessian:
$$\mathcal{H} = D_M^2 \ln Z(M) = \langle \mathbf{n}^{\otimes 4} \rangle_{\rho_M} - \langle \mathbf{n} \otimes \mathbf{n} \rangle_{\rho_M} \otimes \langle \mathbf{n} \otimes \mathbf{n} \rangle_{\rho_M}.$$
 Evaluating the quadratic form of this Hessian on any test matrix $A \in \mathcal{Q}$ gives:
\[\begin{aligned}
A : \mathcal{H} A &= A : \langle \mathbf{n}^{\otimes 4} \rangle_{\rho_M} A - A : \Big( \langle \mathbf{n} \otimes \mathbf{n} \rangle_{\rho_M} \otimes \langle \mathbf{n} \otimes \mathbf{n} \rangle_{\rho_M} \Big) A \\
&= \langle (A : \mathbf{n} \otimes \mathbf{n})^2 \rangle_{\rho_M} - \langle A : \mathbf{n} \otimes \mathbf{n} \rangle_{\rho_M}^2.
\end{aligned}\label{partition11}\]
 This is exactly the variance of the continuous function $\mathbf{n} \mapsto A : \mathbf{n} \otimes \mathbf{n}$ over the sphere. This variance is zero if and only if $A : \mathbf{n} \otimes \mathbf{n}$ is a constant $c$, and this is only possible when $A = c \mathbb{I}_3$. However, this would lead to  $A = 0$ since $\tr(A)   = 0$. Therefore, for any non-zero $A \in \mathcal{Q}$, the variance is strictly positive. This proves that $\ln Z(M)$ is strictly convex over $\mathcal{Q}$, and the operator $\mathcal{H}$ is strictly positive-definite.
Consequently,  it naturally induces an inner product $$\langle X, Y \rangle_\mathcal{H} \triangleq X : \mathcal{H} Y.$$ We can rewrite the right-hand side of \eqref{partition8} using this inner product:
 $$M_Q : A = M_Q : (\mathcal{H} \mathcal{H}^{-1}) A = \mathcal{H} M_Q : \mathcal{H}^{-1} A = \langle M_Q, \mathcal{H}^{-1} A \rangle_\mathcal{H}.$$
 Now, apply the Cauchy-Schwarz inequality to this inner product:
 $$(M_Q : A)^2 \leq \langle M_Q, M_Q \rangle_\mathcal{H} \langle \mathcal{H}^{-1} A, \mathcal{H}^{-1} A \rangle_\mathcal{H}  = (M_Q : \mathcal{H} M_Q) (A : \mathcal{H}^{-1} A).$$
 Rearranging this gives:
 \begin{equation}\non
 A : \mathcal{H}^{-1} A \geq \left( \frac{1}{M_Q : \mathcal{H} M_Q} \right) (M_Q : A)^2.
 \end{equation}
 By \eqref{partition11}, we have 
$$M_Q : \mathcal{H} M_Q = \text{Var}_{\rho_Q}(M_Q : \mathbf{n} \otimes \mathbf{n}),$$
where $\rho_Q$ is (cf.  \eqref{rhoQ}) \begin{equation}\non
\rho_Q(\n)\triangleq \frac{1}{Z(M_Q)}  e^{M_Q : \mathbf{n} \otimes \mathbf{n}}.\end{equation}
So to establish \eqref{partition8}, we need only to show that the scalar quantity $M_Q : \mathcal{H} M_Q$  is uniformly bounded from above by some finite maximum value $V_{\max}$. This will prove the lemma with the uniform constant $C = 1/V_{\max} > 0$.
To this end, recall from \eqref{partition1} that 
 $$M_Q : \mathbf{n} \otimes \mathbf{n} = -\nu_1 x^2 - \nu_2 y^2 - (\mu_1 + \mu_2).$$ Since shifting a random variable by a constant does not change its variance, we have:
 $$M_Q : \mathcal{H} M_Q = \text{Var}_{\rho_Q} (-\nu_1 x^2 - \nu_2 y^2) = \text{Var}_{\rho_Q} (\nu_1 x^2 + \nu_2 y^2).$$

 \textbf{The Asymptotic Variance Bound:}

 We want to show that the variance of the energy $E = \nu_1 x^2 + \nu_2 y^2$ is bounded uniformly for all $\nu_1, \nu_2 > 0$. Because the variance is a continuous function strictly inside the physical domain, it can only blow up if one or both parameters approach infinity. Without loss of generality, assume $\nu_1 \geq \nu_2$. We analyze the two possible asymptotic regimes using Laplace's method. In the following, the notation $o(1)$ denotes an error term that vanishes to zero in the specified limit.

 \textbf{Case 1: Both $\nu_1 \to \infty$ and $\nu_2 \to \infty$.}
 The exponential weight $e^{-(\nu_1 x^2 + \nu_2 y^2)}$ decays infinitely fast everywhere except near the poles $(0,0,\pm 1)$. Near $z=1$, the sphere is asymptotically flat, and the area measure $\ud\mathbf{n}$ approaches the flat Lebesgue measure $\ud x \ud y$. Accounting for both poles, standard Laplace approximation states that the $k$-th moment of the energy approaches the Gaussian integral over the tangent plane up to a vanishing relative error:
 $$ \int_{\mathbb{S}^2} E^k e^{-E} \,\ud\mathbf{n} = \left[ 2 \int_{\mathbb{R}^2} (\nu_1 x^2 + \nu_2 y^2)^k e^{-(\nu_1 x^2 + \nu_2 y^2)} \,\ud x \ud y \right] (1 + o(1)) \quad \text{as } \min(\nu_1, \nu_2) \to \infty. $$
 We can evaluate the integral over $\mathbb{R}^2$ by making the substitution $u = \sqrt{\nu_1} x$ and $v = \sqrt{\nu_2} y$ with Jacobian $\ud x \ud y = \frac{1}{\sqrt{\nu_1 \nu_2}} \ud u \ud v$, and then converting to polar coordinates $(r, \theta)$:
\begin{align*}
 &\int_{\mathbb{R}^2} (\nu_1 x^2 + \nu_2 y^2)^k e^{-(\nu_1 x^2 + \nu_2 y^2)} \,\ud x \ud y \\&= \frac{1}{\sqrt{\nu_1 \nu_2}} \int_{\mathbb{R}^2} (u^2 + v^2)^k e^{-(u^2 + v^2)} \,\ud u \ud v \\
 &= \frac{1}{\sqrt{\nu_1 \nu_2}} \int_0^{2\pi} \ud\theta \int_0^\infty r^{2k} e^{-r^2} r \,\ud r \\
 &= \frac{\pi \Gamma(k+1)}{\sqrt{\nu_1 \nu_2}} = \frac{\pi k!}{\sqrt{\nu_1 \nu_2}}.
\end{align*}
 Substituting this exact result back into our asymptotic relation yields:
 $$ \int_{\mathbb{S}^2} E^k e^{-E} \,\ud\mathbf{n} = \frac{2\pi k!}{\sqrt{\nu_1 \nu_2}} (1 + o(1))\quad \text{as } \min(\nu_1, \nu_2) \to \infty.$$ 
  Evaluating the limits as $\nu_1, \nu_2 \to \infty$ yields:
\begin{align*}
 \lim_{\nu_1, \nu_2 \to \infty} \langle E \rangle_{\rho_Q}   = 1, \quad 
 \lim_{\nu_1, \nu_2 \to \infty} \langle E^2 \rangle_{\rho_Q}   = 2.
\end{align*}
 It then directly follows that the limit of the variance is strictly determined:
\begin{equation}\non
 \lim_{\nu_1, \nu_2 \to \infty} \text{Var}_{\rho_Q}(E) = \lim_{\nu_1, \nu_2 \to \infty} \big( \langle E^2 \rangle_{\rho_Q} - \langle E \rangle_{\rho_Q}^2 \big) =  1.
\end{equation}
 Because this limit exists and evaluates to a finite constant, the variance remains uniformly bounded for all sufficiently large $\nu_1, \nu_2$.

 \textbf{Case 2: $\nu_1 \to \infty$ while $\nu_2 = O(1)$ remains finite.}
 In this regime, the measure concentrates along the great circle $x=0$. Using spherical coordinates 
$$(x,y,z) = (\sin\theta\cos\phi, \sin\theta\sin\phi, \cos\theta),$$ 
the weight is dominated by values near $\phi = \pi/2$ and $3\pi/2$. Expanding around $\phi = \pi/2$ by setting $\phi = \pi/2 + \psi$ for small $\psi$, we have $x \approx -\psi\sin\theta$. The asymptotic energy is $E \approx \nu_1 \psi^2\sin^2\theta  + \nu_2 \sin^2\theta$, and the area element is $\ud\mathbf{n} = \sin\theta \ud\theta \ud\psi$. 

Although this linear expansion is only valid for small $\psi \in [-\epsilon, \epsilon]$, the exponential weight $e^{-\nu_1 \psi^2 \sin^2\theta}$ decays so rapidly as $\nu_1 \to \infty$ that the tails of the integral outside this neighborhood are exponentially small. Extending the limits of integration for $\psi$ to $\pm\infty$ introduces only a negligible error absorbed by the $(1+o(1))$ factor. Accounting for both sides of the circle, Laplace's method gives:
$$ \int_{\mathbb{S}^2} E^k e^{-E} \,\ud\mathbf{n} = 2 \left[ \int_0^\pi \sin\theta  \int_{-\infty}^\infty (\nu_1\psi^2 \sin^2\theta  + \nu_2 \sin^2\theta)^k e^{-\nu_1 \psi^2\sin^2\theta  - \nu_2 \sin^2\theta} \,\ud\psi \,\ud\theta\right](1 + o(1)) $$
as $\nu_1 \to \infty$. Substituting $u = \sqrt{\nu_1} \psi \sin\theta$ gives $\ud\psi = \frac{\ud u}{\sqrt{\nu_1} \sin\theta}$, and the singular $\sin\theta$ term   cancels the one in area element:
\begin{align*}
&  \int_0^\pi \sin\theta  \int_{-\infty}^\infty (\nu_1\psi^2 \sin^2\theta  + \nu_2 \sin^2\theta)^k e^{-\nu_1 \psi^2\sin^2\theta  - \nu_2 \sin^2\theta} \,\ud\psi \,\ud\theta \\
&= \frac{1}{\sqrt{\nu_1}} \int_0^\pi \int_{-\infty}^\infty (u^2 + \nu_2 \sin^2\theta)^k e^{-(u^2 + \nu_2 \sin^2\theta)} \,\ud u \,\ud\theta \triangleq \frac{I_k(\nu_2)}{\sqrt{\nu_1}}.
\end{align*}
Because $\nu_2$ is finite, $I_k(\nu_2)$ is a convergent, strictly positive, bounded integral independent of $\nu_1$. Substituting this back yields:
$$ \int_{\mathbb{S}^2} E^k e^{-E} \,\ud\mathbf{n} = \frac{2}{\sqrt{\nu_1}} I_k(\nu_2) (1+o(1))\text{ as }\nu_1 \to \infty.$$
 When computing the   moments, the $\frac{2}{\sqrt{\nu_1}}$ prefactors cancel entirely:
 $$ \lim_{\nu_1 \to \infty} \langle E^k \rangle_{\rho_Q} = \lim_{\nu_1 \to \infty} \frac{\frac{2}{\sqrt{\nu_1}} I_k(\nu_2) (1+o(1))}{\frac{2}{\sqrt{\nu_1}} I_0(\nu_2) (1+o(1))} = \frac{I_k(\nu_2)}{I_0(\nu_2)}. $$
 Consequently, the   variance  
 $$\text{Var}_{\rho_Q}(E) \xrightarrow{\nu_1 \to \infty}\frac{I_2(\nu_2)}{I_0(\nu_2)} - \left( \frac{I_1(\nu_2)}{I_0(\nu_2)} \right)^2,$$
 which is bounded by a constant depending only on $\nu_2$. 

 Since the exact variance approaches finite limits in all possible asymptotic regimes as $Q$ approaches the physical boundary, there exists a uniform global upper bound $V_{\max} < \infty$ over the entire physical domain $\mathcal{Q}_{phy}$. Setting $C = 1/V_{\max} > 0$ establishes the desired lower bound \eqref{self-concordant}, completing the proof.
 \end{proof}

 \section{Proof of Theorem \ref{main-theorem-1}: Construction of the solution}
 
To construct a solution satisfying the  estimates in Theorem \ref{main-theorem-1}, we follow the truncation scheme used  in \cite{MR3109430,MR3360743}.
 
When building weak solutions to the system, it is convenient to relax the strict physical constraint $Q\in \mathcal{Q}_{phy}$ during the approximate construction. For this reason, we regularize $f$ using the well-known Moreau-Yosida regularization from convex analysis (cf. \cite{MR1422252}), thereby extending the effective domain of $f$ to all of $\mathcal{Q}$. In particular, for $J \in\{1,2,3, \ldots\}$, we define $\widetilde{f}_J$ as:
$$
\tilde{f}_J(Q)\triangleq \min _{A \in \mathcal{Q}}\left(J|A-Q|^2+f(A)\right).
$$
The map $\widetilde{f}_J$ is not guaranteed to be smooth, even if $f$ is $C^{\infty}$ on its effective domain. To ensure its Hessian exists in the classical sense, we require a further mollification. Thus, for fixed $J$ and any $K \in\{1,2,3, \ldots\}$, we define $\widetilde{f}_{J, K} \in C^2(\mathbb{R}^{3\times 3})$ as:
$$
\tilde{f}_{J, K}(Q)\triangleq K^{9} \int_{\mathbb{R}^{3\times 3}} \tilde{f}_J(K(Q-R)) \Phi(R) \,\mathrm{d} R,
$$
where $\Phi \in C_c^{\infty}(\mathbb{R}^{3\times 3}; \mathbb{R}_{+})$ is a standard symmetric mollifier with unit mass $\int \Phi(R) \mathrm{d} R=1$. Finally, we extract the diagonal subsequence $\widetilde{f}_{N, N}$ and define the unified regularization of $f$ to be:
$$
f_N(Q)\triangleq \tilde{f}_{N, N}(Q)=N^{9} \int_{\mathbb{R}^{3\times 3}} \tilde{f}_N(N(Q-R)) \Phi(R) \,\mathrm{d} R.
$$

This choice of regularization is convenient because $f_N$ inherits many of the critical properties of $f$ and converges robustly to its limit. We quote the following established properties of the mollified potential:
 
\begin{proposition}\cite{MR3360743}\label{Moreau}
 For each $N \geq 1$, the regularization $f_N$ of the potential $f$ satisfies:
 \begin{itemize} 
\item $f_N$ is isotropic: $f_N(Q)=f_N(R^TQR)$ for each $R\in SO(3)$.
\item $f_N$ is both $C^{\infty}$ and strictly convex on $\mathcal{Q}$.
\item There is a uniform constant $c_0>0$ such that $-c_0 \leq f_N(X)$ for all $X \in \mathbb{R}^{3\times 3}$ and for all $N \geq 1$.
\item $f_N \leq f_{N+1} \leq f$ on $\mathbb{R}^{3\times 3}$ for $N \geq 1$, and $f_N\xrightarrow{N \rightarrow \infty} \infty $ uniformly on $\mathcal{Q} \backslash \mathcal{Q}_{phy}$.
\item  $\p^\alpha f_N \xrightarrow{N \rightarrow \infty}\p^\alpha  f$ uniformly on compact subsets of $\mathcal{Q}_{phy}$ for $|\alpha|\leq 2$.
\item Its classical first derivative satisfies
\[
c_N^1|X|-c_N^2 \leq\left|\left\langle D_Q f_N(X)\right\rangle\right| \leq C_N^1|X|+C_N^2
\]
for positive constants $c_N^i$ and $C_N^i$ depending on the parameter $N$.
\item The higher derivatives of $f_N$ satisfy the quadratic bound:
$$
\left|\left\langle\frac{\partial^k f_N}{\partial Q^k}(X)\right\rangle\right| \leq C_{N, k}^1|X|^2+C_{N, k}^2 \quad \text{for } k \geq 2.
$$
 \end{itemize} 
\end{proposition}

 Before proceeding to the Galerkin scheme, we rewrite the gradient flow using vector calculus notation and isolate its linear differential operator. Under the convention that $\div Q=\partial_j Q_{ij}=Q_{ij,j}$, we can write \eqref{grad flow E} as:
 \begin{align} 
\p_t Q &= L_1\Delta{Q} +(L_2+L_3) \nabla \div Q-\dfrac{L_2+L_3}{3}(\div \div Q) \mathbb{I}_3 \non\\
&\qquad - \Pi(D_Q f) + 2\alpha Q, \label{gradient-flow-Q}
\end{align}
 where $\Pi(M) = M - \frac{1}{3}\tr(M)\mathbb{I}_3$ denotes the projection onto the space of trace-free symmetric tensors.
 
\begin{lemma}\label{liu-wang}
 \cite{MR3927580} 
 The linear differential operator associated with the elastic energy is  
 \[\mathscr{A}Q\triangleq L_1\Delta{Q} +(L_2+L_3) \nabla \div Q-\dfrac{L_2+L_3}{3}(\div\div Q) \mathbb{I}_3.\]
 Under the assumption that $L_1>0$ and $L_1>|L_2+L_3|$, $\mathscr{A}$ is strongly elliptic and satisfies the strong Legendre condition.
\end{lemma}

\begin{proof}[Proof of Theorem \ref{main-theorem-1}]

 \textbf{Step 1: The Unified Regularized Galerkin Scheme.} Let $N \ge 1$ be a unified parameter. First, we replace the singular potential $f(Q)$ with its Moreau-Yosida mollification $f_N(Q)$. We define the corresponding regularized energy $\mathcal{E}_N(Q)$ whose variational derivative in $\mathcal{Q}$ is:
\begin{equation}\label{finite-dim2}
\delta \mathcal{E}_N(Q) = -\mathscr{A}Q + \Pi\left(D_Q f_N(Q)\right) - 2\alpha Q.
\end{equation}
 Let $V_N = \text{span}\{W_1, W_2, \dots, W_N\}$ be the finite-dimensional space spanned by the first $N$ $L^2$-orthonormal eigenfunctions of the Laplacian $-\Delta$ on the periodic domain $\mathbb{T}^n$. We seek an approximate solution $Q_N(t,x) = \sum_{k=1}^N c_k(t) W_k(x) \in V_N$ satisfying the Galerkin system for all test functions $W \in V_N$:
$$\int_{\mathbb{T}^n} \partial_t Q_N : W = \int_{\mathbb{T}^n} \mathscr{A} Q_N : W - \int_{\mathbb{T}^n} \Pi\left(D_Q f_N(Q_N)\right) : W + 2\alpha \int_{\mathbb{T}^n} Q_N : W.$$
Because $f_N$ is globally smooth and its derivatives are Lipschitz on bounded sets, the Picard-Lindel\"of theorem guarantees a unique local-in-time solution. Because $f_N$ is bounded below, this solution exists globally for $t \in [0, \infty)$.

 \textbf{Step 2: The Basic Energy Estimate.} We test the Galerkin system with $W = \partial_t Q_N \in V_N$. Since this exactly matches the time derivative of the regularized energy, we obtain the standard gradient flow dissipation:
$$\int_{\mathbb{T}^n} |\partial_t Q_N|^2 + \frac{d}{dt} \mathcal{E}_N(Q_N) = 0.$$
Integrating over $t\in [0,T]$ yields:
\begin{equation}\label{galerkin5}
\int_0^T \|\partial_t Q_N\|_{L^2}^2 \,\ud t + \mathcal{E}_N(Q_N(T)) \le \mathcal{E}_N(Q_{0N}) \le \mathcal{E}(Q_0).
\end{equation}
Because $f_N$ is bounded below by a constant $-c_0$, and the linear operator $-\mathscr{A}$ is strongly elliptic by Lemma \ref{liu-wang}, the energy functional bounds the $H^1$ norm. This grants us our first set of uniform bounds, independent of $N$:
\begin{equation}\label{galerkin1}
    \sup_{N \ge 1} \Big( \|Q_N\|_{L^\infty(0, T; H^1 )} + \|\partial_t Q_N\|_{L^2(0, T; L^2 )} \Big) \le C.
\end{equation}

 \textbf{Step 3: The Higher-Order Coercivity Estimate.}
We establish higher-order uniform bounds for $Q_N$ by exploiting the convexity of $f_N$. By \eqref{finite-dim2}, the approximate solution $Q_N \in V_N$ satisfies the finite-dimensional gradient flow: 
\begin{equation}\label{finite-dim1}
\partial_t Q_N = -\mathbb{P}_N \delta \mathcal{E}_N(Q_N)=\mathbb{P}_N \mathscr{A}Q_N - \mathbb{P}_N\Pi\left(D_Q f_N(Q_N)\right) + 2\alpha Q_N,
\end{equation}
where $\mathbb{P}_N$ is the $L^2$-orthogonal projection onto $V_N$. Because $V_N$ is closed under differentiation, the test function $W = -\Delta Q_N$ belongs to $V_N$. Using the fact that $\Delta Q_N$ is trace-free (allowing us to drop the projection $\Pi$) and applying the chain rule $\nabla(D_Q f_N(Q_N)) = D^2_Q f_N(Q_N) \nabla Q_N$, we integrate by parts:
\begin{align}\label{H1-estimate-packed}
\frac{1}{2} \frac{d}{dt} \int_{\TT^n} |\nabla Q_N|^2 &= - \int_{\TT^n} \p_t Q_N : \Delta Q_N \notag \\
&= - \int_{\TT^n} \mathscr{A} Q_N : \Delta Q_N + \int_{\TT^n} D_Q f_N(Q_N) : \Delta Q_N - 2\alpha \int_{\TT^n} Q_N : \Delta Q_N \notag \\
&= -\int_{\TT^n} \Big[ L_1 |\Delta Q_N|^2 + (L_2+L_3) |\nabla \div Q_N|^2 \Big] \notag \\
&\quad - \int_{\TT^n} D^2_Q f_N(Q_N) : (\nabla Q_N)^{\otimes 2} + 2\alpha \int_{\TT^n} |\nabla Q_N|^2.
\end{align}
Because $f_N$ is strictly convex (cf. Proposition \ref{Moreau}), its Hessian is positive semi-definite everywhere, meaning $D^2_Q f_N(Q_N) : (\nabla Q_N)^{\otimes 2} \ge 0$.

 To obtain regularity in time, we take the time derivative of the Galerkin equation. Because the initial data satisfies $\|\delta \mathcal{E}(Q_0)\|_{L^2} < \infty$, the initial time derivative is uniformly bounded in $L^2$:
 $$ \|\partial_t Q_N(0)\|_{L^2} = \|-\mathbb{P}_N \delta \mathcal{E}_N(Q_{0N})\|_{L^2} \le \|\delta \mathcal{E}(Q_0)\|_{L^2} < \infty. $$
Taking the time derivative of \eqref{finite-dim1} and testing with $\partial_t Q_N \in V_N$, the projection $\mathbb{P}_N$ acts as the identity. Applying the chain rule yields:
\begin{align}\label{finite-dim6}
\frac{1}{2} \frac{d}{dt} \int_{\TT^n} |\partial_t Q_N|^2 &= \int_{\TT^n} \partial_{tt} Q_N : \partial_t Q_N \notag \\
&= - \int_{\TT^n} \Big[ L_1 |\nabla \partial_t Q_N|^2 + (L_2+L_3) |\div \partial_t Q_N|^2 \Big] \notag \\
&\quad - \int_{\TT^n} D^2_Q f_N(Q_N) : (\partial_t Q_N)^{\otimes 2} + 2\alpha \int_{\TT^n} |\partial_t Q_N|^2.
\end{align}
Using the strong Legendre condition ($L_1 > |L_2+L_3|$, cf. Lemma \ref{liu-wang}), the spatial gradient terms are strictly coercive. Moving the negative dissipative terms to the left-hand side yields the differential inequality:
\begin{equation}\label{time-coercive}
\frac{1}{2} \frac{d}{dt} \|\partial_t Q_N\|_{L^2}^2 + C_0 \|\nabla \partial_t Q_N\|_{L^2}^2 + \int_{\TT^n} D^2_Q f_N(Q_N) : (\partial_t Q_N)^{\otimes 2} \le 2\alpha \|\partial_t Q_N\|_{L^2}^2.
\end{equation}

The differential inequalities \eqref{H1-estimate-packed} and \eqref{time-coercive} isolate the growth of the system to the $2\alpha$ mass terms. Applying Gronwall's inequality independently to both estimates on the interval $[0, T]$, we conclude that the norms cannot blow up. This grants us the following $N$-independent uniform bounds:
\[\begin{aligned}\label{galerkin3}
&\sup_{N \ge 1} \left( \|\nabla Q_N\|_{L^\infty(0, T; L^2 )}+ \|\nabla^2 Q_N\|_{L^2(0, T; L^2 )} \right)\\
&+ \sup_{N \ge 1} \left(\|\partial_t Q_N\|_{L^\infty(0, T; L^2 )} + \|\partial_t Q_N\|_{L^2(0, T; H^1 )} \right) \le C(T, Q_0).
\end{aligned}
 \]
Testing \eqref{finite-dim1} with  $\Delta Q_N$, and   utilizing  convexity of $f_N$ along with  \eqref{galerkin3}, we find
\begin{equation}\label{galerkin6}
\|\nabla^2 Q_N\|_{L^\infty(0, T; L^2 )}\leq   C \|\partial_t Q_N\|_{L^\infty(0, T; L^2 )}  \le C(T, Q_0).
\end{equation}

 \textbf{Step 4: Compactness and the Nonlinear Bound.} With bounds \eqref{galerkin1} and \eqref{galerkin3}, the Aubin--Lions--Simon compactness lemma allows us to extract a strongly convergent subsequence (still denoted $N$). As $N \to \infty$:
\begin{equation}\label{galerkin2}
\begin{aligned}
    Q_N &\to Q &&\text{strongly in } C([0,T]; L^2) \cap L^2(0, T; H^1) \text{ and pointwise a.e.}, \\
    \partial_t Q_N &\rightharpoonup \partial_t Q &&\text{weakly in } L^2(0, T; H^1) \text{ and weak-* in } L^\infty(0, T; L^2),\\
      \nabla^2 Q_N &\rightharpoonup \nabla^2 Q &&  \text{   weak-* in } L^\infty(0, T; L^2).
\end{aligned}
\end{equation}
These convergence results justify  \eqref{infty2-solu}. 
By \eqref{galerkin3} and \eqref{galerkin6}, we have 
\begin{equation}\label{galerkin4}
\left\| \mathbb{P}_N \Pi\left(D_Q f_N(Q_N)\right) \right\|_{L^\infty(0,T; L^2)} \le C.
\end{equation}

 \textbf{Step 5: The limit $Q$ is physical almost everywhere.}
We will show that the limit solution strictly satisfies $Q(t,x) \in \mathcal{Q}_{phy}$ almost everywhere. By \eqref{galerkin5},
$$\sup_{t \in [0,T]} \int_{\mathbb{T}^n} f_N(Q_N(t,x)) \,\ud x \le M$$
for some constant $M$ independent of $N$. Since the regularized potential $f_N(Q) \ge -c_0$, the shifted sequence $f_N(Q_N(t,x)) + c_0$ is strictly non-negative. Integrating over space and time yields:
$$\int_0^T \int_{\mathbb{T}^n} \Big( f_N(Q_N) + c_0 \Big) \,\ud x \ud t \le T(M + c_0 |\mathbb{T}^n|) < \infty.$$
Suppose, for the sake of contradiction, that there exists a set $S \subset \mathbb{T}^n \times (0,T)$ of strictly positive measure where the limit $Q(t,x)$ lies outside the physical domain $\mathcal{Q}_{phy}$. For almost every $(t,x) \in S$, the sequence $Q_N(t,x)$ converges to a target $Q$ outside $\mathcal{Q}_{phy}$. By Proposition \ref{Moreau}, $f_N \to \infty$ outside $\mathcal{Q}_{phy}$, so:
$$\liminf_{N \to \infty} \Big( f_N(Q_N(t,x)) + c_0 \Big) = +\infty \quad \text{for almost every } (t,x) \in S.$$
Because $f_N(Q_N) + c_0 \ge 0$, we apply Fatou's Lemma over the space-time domain:
$$\int_0^T \int_{\mathbb{T}^n} \liminf_{N \to \infty} \Big( f_N(Q_N) + c_0 \Big) \,\ud x \ud t \le \liminf_{N \to \infty} \int_0^T \int_{\mathbb{T}^n} \Big( f_N(Q_N) + c_0 \Big) \,\ud x \ud t < \infty.$$ 
Since the left side evaluates to $+\infty$ over the set $S$, this leaves us with a strict contradiction unless $S$ has measure zero. Therefore, $Q(t,x) \in \mathcal{Q}_{phy}$ almost everywhere.

 \textbf{Step 6: Limit of the regularized coercivity estimates.}
 We now pass to the limit $N \to \infty$ in the Hessian integrals to establish \eqref{hessian-space}. Returning to the coercive bounds \eqref{H1-estimate-packed} and \eqref{time-coercive}, we integrate over time $t \in [0,T]$. Because the boundary terms and non-Hessian integrals are uniformly bounded by the previous steps, we deduce that:
$$ \int_0^T \int_{\mathbb{T}^n} D^2_Q f_N(Q_N) : (\nabla Q_N)^{\otimes 2} \,\ud x \ud t \le C, \qquad \int_0^T \int_{\mathbb{T}^n} D^2_Q f_N(Q_N) : (\partial_t Q_N)^{\otimes 2} \,\ud x \ud t \le C. $$

 From Step 5, $Q(t,x)$ lies strictly in the interior of the physical domain almost everywhere. Since $f_N \to f$ in $C^2_{loc}(\mathcal{Q}_{phy})$ and $Q_N \to Q$ almost everywhere, the Hessians converge pointwise:
$$D^2_Q f_N(Q_N) \to D^2_Q f(Q) \quad \text{a.e. in } \mathbb{T}^n \times (0,T).$$
Because the matrices $D^2_Q f_N(Q_N)$ are positive semi-definite and converge pointwise, and $\nabla Q_N \rightharpoonup \nabla Q$ as well as $\partial_t Q_N \rightharpoonup \partial_t Q$ weakly in $L^2(L^2)$, we invoke standard weak lower semicontinuity theorems for integral functionals  to conclude:
$$\int_0^T \int_{\mathbb{T}^n} D^2_Q f(Q) : (\nabla Q)^{\otimes 2} \,\ud x \ud t \le \liminf_{N \to \infty} \int_0^T \int_{\mathbb{T}^n} D^2_Q f_N(Q_N) : (\nabla Q_N)^{\otimes 2} \,\ud x \ud t \le C,$$
$$\int_0^T \int_{\mathbb{T}^n} D^2_Q f(Q) : (\partial_t Q)^{\otimes 2} \,\ud x \ud t \le \liminf_{N \to \infty} \int_0^T \int_{\mathbb{T}^n} D^2_Q f_N(Q_N) : (\partial_t Q_N)^{\otimes 2} \,\ud x \ud t \le C.$$

 This confirms \eqref{hessian-space}. Passing the linear terms to the limit in the standard distributional sense establishes that \eqref{strong solu} is satisfied almost everywhere.
 
 Finally  
 by setting $A = \partial_i Q$ in \eqref{self-concordant}, we obtain:
 \begin{equation} \label{self-concordant1} 
 \partial_i Q : D^2_Q f(Q) \partial_i Q \geq  C |D_Q f(Q) : \p_i Q|^2=|\partial_i (f(Q))|^2 \qquad \text{ for a.e. }x.
 \end{equation}
This together with    \eqref{hessian-space} establishes   \eqref{wellposed1}.\end{proof}


 \section{Proof of Theorem \ref{thm-contact-set}: Analysis  of the contact set}
 In this section we give a   proof of Theorem \ref{thm-contact-set}.
  In two dimensions, $H^1$ spaces strictly fail to embed into $C^0$, allowing for potential logarithmic singularities. However, we will show that such a singularity is incompatible with the Hölder continuity of $Q$.

\subsection{The 2D case ($\Omega \subset \mathbb{R}^2$)}

Assume for the sake of contradiction that $0\in\mathcal{C}(Q)$. Then the minimum eigenvalue reaches its absolute limit: $\lambda_1(Q(0)) = -1/3$. 
Because $Q \in H^2(\Omega)$, Sobolev embedding in 2D gives $H^2(\Omega) \hookrightarrow C^{0,\gamma}(\Omega)$ for any Hölder exponent $\gamma \in (0,1)$. By the Lipschitz continuity of eigenvalues and the equivalence of norms, we have:
\[\begin{aligned}
0 \leq  \lambda_1(Q(x))+\frac{1}{3} &= | \lambda_1(Q(x))-\lambda_1(Q(0))| \\
 &\leq \|Q(x)-Q(0)\|_{\mathrm{op}} \leq |Q(x)-Q(0)|  \leq C|x|^\gamma.
\end{aligned} \label{lower-bound-new-1}\]
It follows from the asymptotic lower bound \eqref{lower-bound-new} that the scalar function $w(x) \triangleq f(Q(x))$ satisfies:
$$w(x) \geq -\frac{1}{2}\ln\Big(\lambda_1(Q(x))+\frac{1}{3}\Big) - C_1 \geq -\frac{1}{2}\ln(C|x|^\gamma) - C_1 = \frac{\gamma}{2} \big|\log |x|\big| - C_2.$$
Consequently, as $x\to 0$, $w(x)\to  +\infty$. The Hölder continuity provides a rigorous pointwise lower bound on the growth rate:
\begin{equation}\label{average-lower1-Q}
w(x) \geq \frac{\gamma}{2}\big|\log|x|\big| - C_2.
\end{equation}
Define the spherical average of $w$ at radius $r$:
$$\bar{w}(r) = \frac{1}{2\pi} \int_0^{2\pi} w(r, \theta) \,\ud\theta.$$
Since $w(x)$ satisfies the lower bound \eqref{average-lower1-Q} near the origin, its angular average inherits the exact same logarithmic growth:
\begin{equation}\label{average-lower2-Q}
\bar{w}(r) \geq \frac{\gamma}{2} |\log r| - C_2.
\end{equation}
Since $w \in H^1(B_R)$, it is absolutely continuous on almost all rays. We can relate the radial derivative of the average to the full spatial gradient:
$$\bar{w}'(r) = \frac{1}{2\pi} \int_0^{2\pi} \frac{\partial w}{\partial r}(r, \theta) \,\ud\theta.$$
Using the Cauchy-Schwarz inequality on the angular integral:
$$|\bar{w}'(r)|^2 \leq \frac{1}{2\pi} \int_0^{2\pi} \left| \frac{\partial w}{\partial r} \right|^2 \,\ud\theta \leq \frac{1}{2\pi} \int_0^{2\pi} |\nabla w|^2 \,\ud\theta.$$
Multiplying by $r$ and integrating from an arbitrarily small $r_1 > 0$ up to a fixed $R$:
$$\int_{r_1}^R |\bar{w}'(r)|^2 r \,\ud r \leq \frac{1}{2\pi} \int_{r_1}^R \int_0^{2\pi} |\nabla w|^2 r \,\ud\theta \,\ud r \leq \frac{1}{2\pi} \|\nabla f(Q) \|_{L^2(B_R)}^2 \triangleq  M < \infty.$$
Now we investigate how fast $\bar{w}(r_1)$ can possibly grow as $r_1 \to 0$ based on this $H^1$ energy bound. By Cauchy-Schwarz again:
\begin{align*}
\bar{w}(r_1) &= \bar{w}(R) - \int_{r_1}^R \bar{w}'(r) \,\ud r \\
 &\leq \bar{w}(R) + \int_{r_1}^R |\bar{w}'(r)| \sqrt{r} \frac{1}{\sqrt{r}} \,\ud r \\
 &\leq \bar{w}(R) + \left( \int_{r_1}^R |\bar{w}'(r)|^2 r \,\ud r \right)^{1/2} \left( \int_{r_1}^R \frac{1}{r} \,\ud r \right)^{1/2}\\
 &\leq \bar{w}(R) + \sqrt{M} (\log R - \log r_1)^{1/2}.
\end{align*}
As $r_1 \to 0$, this inequality indicates   that $\bar{w}(r_1)$ can grow  at most  at the rate of $\sqrt{|\log r_1|}$. However, \eqref{average-lower2-Q} guarantees that $\bar{w}(r_1)$  must grow at the strictly faster rate of $\frac{\gamma}{2} |\log r_1|$. 
Because $|\log r_1|$ grows infinitely faster than $\sqrt{|\log r_1|}$ as $r_1 \to 0$, the two bounds are fundamentally incompatible. The assumption $\lambda_1(Q(0))=-1/3$ directly contradicts the proven regularity $\nabla f(Q) \in L^2(\Omega)$. Therefore, the contact set $\mathcal{C}(Q)$ must be   empty in 2D.
 
\subsection{The 3D case ($\Omega \subset \mathbb{R}^3$)}
 In three dimensions, we can no longer guarantee the contact set is entirely empty, but we can strictly bound its size. We follow the dimensional strategy of \cite{MR3451881}, but we replace their distance/convexity argument (which relies on a simplified toy model) with our rigorous spherical averaging method to handle the true singular potential.

\begin{lemma}\label{lin-ref} \cite[Lem 2.1.1]{MR2030862}
 Let $g \in L_{\mathrm{loc}}^1 (\mathbb{R}^n )$. Suppose $0 \leq s<n$, and define the set of points where the fractional density is positive:
$$
M_s \triangleq \left\{x \in \mathbb{R}^n \;\middle|\; \limsup _{r \rightarrow 0} \frac{1}{r^s} \int_{B(x, r)}|g| \,\ud y > 0 \right\}.
$$
Then the $s$-dimensional Hausdorff measure is zero: $\mathcal{H}^s\left(M_s\right)=0$.
 (In standard Sobolev applications, this is applied with $s=n-2$ or $s=n-2+\epsilon$.)
\end{lemma}

 Because $Q \in H^2(\Omega) \hookrightarrow C^{0,1/2}(\Omega)$ in 3D, we automatically have pointwise continuity everywhere: $\lim_{r \rightarrow 0}(  Q)_{x, r} =   Q(x)$, where 
 $$(g)_{x, r} \triangleq \fint_{B(x, r)} g(y) \, \mathrm{d}y = \frac{1}{\vert{}B(x, r)\vert{}} \int_{B(x, r)}   g(y) \, \mathrm{d}y.$$
  We define the following subsets of $\Omega$ based on Lebesgue differentiation properties:
 \begin{subequations}
\begin{align}
&U_1 \triangleq \left\{x \in \Omega \;\middle|\; \lim _{r \rightarrow 0} \fint_{B(x, r)} | \nabla Q - (\nabla Q )_{x, r} |^2 \,\ud y=0 \right\},\\
&U_2 \triangleq \left\{x \in \Omega \;\middle|\;  \lim _{r \rightarrow 0}(\nabla Q )_{x, r}=\nabla Q (x),\,|\nabla Q (x)|<\infty \right\}.
\end{align}
\end{subequations}
Using Poincaré's inequality on balls and the fact that $\nabla^2 Q \in L^2(\Omega)$, we apply Lemma \ref{lin-ref} (with $g = |\nabla^2 Q|^2$ and $s = n-2 = 1$) to find:
$$
\begin{aligned}
\mathcal{H}^{1}\left(\Omega \setminus U_1\right) &= \mathcal{H}^{1}\left(\left\{x \in \Omega \;\middle|\; \limsup _{r \rightarrow 0} \fint_{B(x, r)} | \nabla Q -  (\nabla Q )_{x, r} |^2 \,\ud y>0\right\}\right) \\
& \leq \mathcal{H}^{1}\left( \left\{x \in \Omega \;\middle|\; \limsup _{r \rightarrow 0} r^{-1} \int_{B(x, r)} | \nabla^2 Q |^2 \,\ud y>0\right\}\right)=0.
\end{aligned}
$$
 Since $\nabla Q\in H^1$, standard capacity theory indicates that its Lebesgue averages $(\nabla Q )_{x, r}$ converge pointwise outside a set $E$ with  ${\rm Cap_2}(E)=0$, cf. \cite[Thm 4.19]{MR3409135}.  This together with \cite[Thm 4.17]{MR3409135} implies $\mathcal{H}^{1+\epsilon}(\Omega \setminus U_2)=0$ for any $\epsilon>0$.

 To analyze the blowup of the entropy, we define a set $U_4$ based on the gradient of the potential $f(Q)$:
$$
U_4 = \left\{x \in \Omega \;\middle|\; \limsup _{r \rightarrow 0} r^{-1} \int_{B(x, r)} |\nabla f(Q )|^2 \,\ud y > 0 \right\}.
$$
Since $f(Q) \in H^1(\Omega)$, applying Lemma \ref{lin-ref} directly guarantees that $\mathcal{H}^1(U_4)=0$. 
Finally, we define the intersection of the contact set with our "well-behaved" Lebesgue points:
\begin{equation}\label{contact1}
U_3 \triangleq \left\{x \in U_1 \cap U_2 \mid f(Q(x))=\infty\right\} = U_1 \cap U_2 \cap \mathcal{C}(Q).
\end{equation}
 If we can prove the claim that $U_3 \subset U_4$, then the physical contact set $$\mathcal{C}(Q)\subset (\Omega \setminus U_1) \cup (\Omega \setminus U_2) \cup U_4.$$  Because all three of these covering sets have vanishing $\mathcal{H}^{1+\epsilon}$ measure, we can conclude that the contact set $\mathcal{C}(Q)$ has Hausdorff dimension at most 1.

\subsection{Proof of the claim $U_3\subset U_4$}

We proceed by contradiction. Suppose there exists a point $x_0 \in U_3$ such that $x_0 \notin U_4$. The Sobolev embedding $H^2(\Omega) \hookrightarrow C^{0,1/2}(\Omega)$ in 3D provides the sharp Hölder bound:
$$|Q(x) - Q(x_0)| \le C|x - x_0|^{1/2}.$$
Because $f(Q(x_0)) = \infty$, we know $\lambda_1(Q(x_0)) = -1/3$. It follows from the asymptotic lower bound \eqref{lower-bound-new} and the estimate in \eqref{lower-bound-new-1}  that the scalar function $w(x) \triangleq f(Q(x))$ satisfies 
\begin{align*}
w(x)&   \geq -\frac{1}{2}\ln\Big(\lambda_1(Q(x))+\frac{1}{3}\Big) - C_1 \\
&\ge -\frac{1}{2}\ln(C|x-x_0|^{1/2}) = \frac{1}{4} \big|\log|x-x_0|\big| - C_2
\end{align*}
near $x_0$.
Consequently, the spherical average $$\bar{w}(\rho) = \fint_{\partial B_\rho(x_0)} w(y) \,\ud S$$ must grow at least at the rate of:
\begin{equation} \label{average-lower-3d}
\bar{w}(\rho) \ge \frac{1}{4} |\log \rho| - C_2.
\end{equation}
To exploit the contradiction, we introduce the volume and surface gradient energies:
\begin{align*}
\varepsilon(r) &\triangleq  r^{-1}\int_{B_r(x_0)} |\nabla w|^2 \,\ud y,\\
\phi(\rho) &\triangleq \int_{\partial B_\rho(x_0)} |\nabla w|^2 \,\ud S.
\end{align*}
Because we assumed $x_0 \notin U_4$, the volumetric density must vanish in the limit:
\[ \varepsilon(r) \to 0 \quad \text{as } r \to 0. \] 
We calculate the radial derivative of the spherical average. Applying the Cauchy-Schwarz inequality  on the surface of the sphere yields:
$$\bar{w}'(\rho) = \fint_{\partial B_\rho} \frac{\partial w}{\partial \nu} \,\ud S \leq  \sqrt{\frac{1}{4\pi \rho^2} \int_{\partial B_\rho} |\nabla w|^2 \,\ud S} = \frac{1}{\sqrt{4\pi}} \frac{\sqrt{\phi(\rho)}}{\rho}.$$
We integrate this inequality from a very small radius $r_1$ up to $R$:
$$\bar{w}(r_1) - \bar{w}(R) \le \int_{r_1}^R |\bar{w}'(\rho)| \,\ud\rho \le \frac{1}{\sqrt{4\pi}} \int_{r_1}^R \frac{\sqrt{\phi(\rho)}}{\rho} \,\ud\rho.$$
To accurately evaluate this integral, we decompose the interval $[r_1, R]$ into dyadic shells. Let $R_k = 2^{-k}R$. On a single shell $A_k = [R_{k+1}, R_k]$, we apply the standard 1D Cauchy-Schwarz inequality to the integral:
\begin{align*}
\int_{R_{k+1}}^{R_k} \frac{\sqrt{\phi(\rho)}}{\rho} \,\ud\rho &\le \left( \int_{R_{k+1}}^{R_k} \phi(\rho) \,\ud\rho \right)^{1/2} \left( \int_{R_{k+1}}^{R_k} \frac{1}{\rho^2} \,\ud\rho \right)^{1/2}.
\end{align*}
Notice that the square of the first factor is bounded by the volume integral over the entire ball $B_{R_k}$, which is precisely $R_k \varepsilon(R_k)$. The square of the second factor evaluates exactly to $\frac{1}{R_{k+1}} - \frac{1}{R_k} = \frac{1}{R_k}$. Therefore, the product simplifies to:
\begin{equation}\non
\int_{R_{k+1}}^{R_k} \frac{\sqrt{\phi(\rho)}}{\rho} \,\ud\rho \le \sqrt{R_k \varepsilon(R_k)} \sqrt{\frac{1}{R_k}} = \sqrt{\varepsilon(R_k)}.
\end{equation}
We sum this over all $N$ dyadic shells between $r_1$ and $R$. The total number of shells is $N \approx \log_2(R/r_1) = \frac{1}{\ln 2} (\ln R - \ln r_1)$:
$$\bar{w}(r_1) \le \bar{w}(R) + \frac{1}{\sqrt{4\pi}} \sum_{k=0}^N \sqrt{\varepsilon(R_k)}.$$
Because $\varepsilon(r) \to 0$, for any arbitrarily small    $\delta > 0$, there exists an index $K$ such that for all $k \ge K$, we have $\sqrt{\varepsilon(R_k)} \le \delta$. We can split the sum at $K$:
$$\sum_{k=0}^N \sqrt{\varepsilon(R_k)} \le C_K + \sum_{k=K}^N \delta \le C_K + \delta N \le C_K + \frac{\delta}{\ln 2} \big|\log r_1\big|.$$
This implies that for any $\delta > 0$, the spherical average $\bar{w}(r_1)$ grows strictly slower than $\delta|\log r_1|$ as $r_1 \to 0$.  By choosing $\delta$ sufficiently small (e.g., $\delta < \frac{\ln 2}{4}$), this strictly contradicts the established lower bound \eqref{average-lower-3d} that $\bar{w}(\rho) \ge \frac{1}{4}|\log \rho|$.

Therefore, our initial assumption must be false. The point $x_0$ must belong to $U_4$, establishing that $U_3 \subset U_4$ and completing the proof.

 \section*{Acknowledgments}
 Yuning Liu acknowledges the support of the National Natural Science Foundation of China under Grant No.~12571224.
 Xiang Xu acknowledges the support of the National Natural Science Foundation 
under Grant No.~NSF DMS-2307525.


\end{document}